\documentclass[12pt]{amsart}
\usepackage{etex}
\usepackage[english]{babel}
\usepackage{latexsym}
\usepackage{amsbsy}
\usepackage{amssymb}
\usepackage{amsfonts, amsmath}
\usepackage{mathtools}
\usepackage{csquotes}
\usepackage{pst-tree}
\usepackage{multirow}
\usepackage{blkarray}
\usepackage{tabularx}
\usepackage{arydshln,leftidx,mathtools}
\usepackage{ifthen}
\usepackage{tikz}
\usepackage{tikz-cd}
\usepackage{float}
\usepackage{picins}
\usepackage{amsbsy}
\usetikzlibrary{matrix}
\usetikzlibrary{calc}
\usetikzlibrary{fit}
\usepackage{amsmath,amsfonts,amssymb}
\usepackage{mathrsfs,eurosym}
\usepackage{pst-plot,pst-eucl}
\usepackage{amsfonts}
\usepackage{amssymb}
\usepackage{amsmath}

\usepackage{latexsym}
\usepackage{amsbsy}
\usepackage{alltt}
\usepackage{fontenc}
\usepackage{newlfont}
\usepackage{xlop}
\usepackage{graphicx}
\usepackage{yhmath}
\usepackage{mathdots}

\usepackage{MnSymbol}
\usepackage{amsthm}
\usepackage{url}
\usepackage{enumerate}

\newtheorem{thm}{Theorem}[section]
 \newtheorem{cor}[thm]{Corollary}
 \newtheorem{lem}[thm]{Lemma}
 \newtheorem{prop}[thm]{Proposition}

 \theoremstyle{definition}
 \newtheorem{dfn}[thm]{Definition}
 \newtheorem{ex}[thm]{Example}

 \theoremstyle{remark}
 \newtheorem{rmk}[thm]{Remark}

\usepackage{cite}

\numberwithin{equation}{section}

\newcommand{\e}{{\mathrm{e}}}

\makeatletter
\@namedef{subjclassname@2020}{%
  \textup{2020} Mathematics Subject Classification}
\makeatother
\title{$\mathcal{P}$-canonical forms and Drazin inverses}
\author{\bf M. MOU\c{C}OUF}
\date{}


\begin{document}
\begin{abstract}
In this paper, $\mathcal{P}$-canonical forms of $(A^{k})_{k}$ (or simply of the matrix $A$) are defined and some of their properties are proved. It is also shown how we can deduce from them many interesting informations about the matrix $A$. In addition, it is proved that the  $\mathcal{P}$-canonical forms of $A$ can be written as a sum of two parts, the geometric and the non-geometric parts of $A$, and that the $\mathcal{P}$-canonical form of the Drazin inverse $A_{d}$ of $A$ can be deduced by simply plugging $-k$ for $k$ in the geometric part of $A$. Finally, several examples are provided to illustrate the obtained results.
\end{abstract}
\maketitle
\begin{center}
{\footnotesize Department of Mathematics, Faculty of Science, Chouaib Doukkali University,\\Morocco\\
Email: moucouf@hotmail.com}
\end{center}
\section{Introduction}
\label{sec1}
Let $A$ be a nonsingular matrix over a field $F$. It is a well known fact that, for many numerical examples of matrices $A$, replacing $k$ with $-k$ in certain forms of $A^{k}$, one can obtain the $k$th power of the inverse $A^{-1}$ of $A$. However, this fact is not proven in general. The problem here is that the $k$th power of a matrix can be represented in several forms. In fact, It can happen that in certain forms of $A^{k}$ it is not even possible to substitute $k=-1$ and, in other forms we can substitute $k=-1$ but we does not obtain $A^{-1}$. For example, Let $A$ be an $r$-circulant matrix. The expressions of $A^{k}$ given in Theorem 4.1 of~\cite{Mou1} and Theorem 3.1 of~\cite{Mou2} do not provide $A^{-1}$ when we plug into them $k=-1$.
\\In this paper, we consider an arbitrary matrix $A$, singular or nonsingular, with entries in a field $F$ and we prove that if we plug $-k$ for $k$ into the geometric part of $(A^{k})_{k}$ we get the $k$th power of the Drazin inverse of $A$. In order to avoid any confusion that may arise by using this plugging-in operation, we begin by showing that the representations of $(A^{k})_{k}$ into them we plug $-k$ for $k$, are canonical.
\\An element $a$ of an associative ring is said to have a Drazin inverse~\cite{Draz} if there exists an element $b$, written $b=a_{d}$, such that
\begin{equation*}
a^{k}ba=a^{k},\hspace{1.5cm} ab=ba,\hspace{1.5cm} bab=b
\end{equation*}
for some nonnegative integer $k$. It is well known that any element of any associative ring has at most one Drazin inverse (see~\cite{Draz}).
\\The theory of Drazin inverse has been extensively studied and successfully applied in many fields of science such as functional analysis, matrix
computations, combinatorial problems, numerical analysis, statistics, population models, differential equations, Markov chains, control theory, and cryptography, etc.~\cite{Camp, Camp1, Ben, Zhang, Liu, Lev}. For this reason, many interesting properties of the Drazin inverse have been obtained~\cite{Draz, Cvet, Wei, Camp2, Liu} and a variety of direct and iterative methods have been developed for computation of this type of generalized inverse~\cite{Grev, Ros, Wilk, Camp3, Zhang, Cvet, Wei, Sol, Kyrc, Yu}.
\\It should be noted that the method presented herein provides a closed-form formula for the $k$th power of the Drazin inverse $A_{d}$ of $A$ and gives other interesting information on the matrix $A$ such as the minimal polynomial and the Jordan-Chevalley decomposition of $A$. But for this, a closed-form formula for the $k$th power of the matrix $A$ is required and this can be done by any of the well-known methods such as Kwapisz's method~\cite{Kwap}.
\\The organization of this paper is as follows. In Section 2, some algebraic results are established for the $F$-algebra of all linear recurrence sequences over $F$ whose characteristic polynomials split over $F$, and then, the definitions of the geometric part and the non-geometric part of a linear recurrence sequence are given. Results obtained there will be used in Section 3 to derive similar results for the set of all sequences of matrices, with coefficients in $F$, satisfying linear homogeneous recurrence relations with constant coefficients in $F$. In addition, some interesting properties are shown about the $\mathcal{P}$-canonical forms of matrices defined in this section. Lastly, section 4 shows that the $\mathcal{P}$-canonical forms of the Drazin inverse of a matrix $A$ can be deduced from those of the matrix $A$ by a simple plugging-in operation.
\subsection{Notations}\hfill\\
Throughout the paper, we use the following notations:
\begin{itemize}
\item $F$ is a field and $\mathcal{C}_{F}$ is the set of all sequences $\pmb{s}=(s_{k})_{k\geqslant 0}$ over $F$. It is well known that $\mathcal{C}_{F}$ is an $F$-algebra under componentwise addition, multiplication and scalar multiplication.
\item Sequences in this paper are written in bold symbol.
\item $\pmb{\Lambda}_{i}$, $i\geqslant 0$, is the element $(\binom{k}{i})_{k\geq 0}$ of $\mathcal{C}_{F}$.
\item $\pmb{\widetilde{\Lambda}}_{i}$, $i\geqslant 0$, is the element $(\binom{-k}{i})_{k\geq 0}$ of $\mathcal{C}_{F}$.
\item $\mathcal{T}$ is the set $\{\pmb{\Lambda}_{i}/i\geqslant 0\}$.
\item $\pmb{\Gamma}$ is the element $(0,1,2,\ldots)$ of $\mathcal{C}_{F}$.
\item $\mathcal{H}$ is the set $\{\pmb{\Gamma}^{i}/i\geqslant 0\}$.
\item If $\mathcal{R}$ is a subring of a ring $\mathcal{D}$ and $L$ is subset of $\mathcal{D}$, then $\mathcal{R}\langle L\rangle$ denotes the submodule spanned by $L$.
\item If $\mathcal{R}$ is a subalgebra of an algebra $\mathcal{K}$ and $L$ is a subset of $\mathcal{K}$, then $\mathcal{K}[L]$ denotes the subalgebra of obtained by adjoining to $\mathcal{K}$ the set $L$.
\item $M_{q}(G)$ is the set of $q\times q$ matrices over an $F$-algebra $G$.
\item $\mathcal{S}^{\ast}=\{\pmb{\lambda}=(\lambda^{k})_{k\geqslant 0}/\lambda\in F, \lambda\neq 0\}$ the set of all nonzero geometric sequences.
\item $\mathcal{S}^{\circ}=\{\pmb{0}_{i}\in \mathcal{C}_{F}/i\in \mathbb{N}, \pmb{0}_{i}(k)=\delta_{ik}\}$.
\item $\mathcal{S}=\mathcal{S}^{\ast}\cup \mathcal{S}^{\circ}$.
\item $F_{\mathcal{S}}$ denotes the $F$-vector spaces spanned by $\mathcal{S}$. It is well known that $F_{\mathcal{S}}$ is the set of all linear recurrence sequences over $F$ whose characteristic polynomials are of the form $X^{m}P(X)$ where the polynomials $P(X)$ are square-free and split over $F$ with nonzero constant terms.
\item $F_{\mathcal{S}^{\ast}}$ denotes the $F$-vector spaces spanned by $\mathcal{S}^{\ast}$. It is well known that $F_{\mathcal{S}^{\ast}}$ is the set of all linear recurrence sequences over $F$ whose characteristic polynomials are square-free and split over $F$ with nonzero constant terms.
\item $F_{\mathcal{S}^{\circ}}$ denotes the $F$-vector spaces spanned by $\mathcal{S}^{\circ}$. It is well known that $F_{\mathcal{S}^{\circ}}$ is the set of all linear recurrence sequences over $F$ whose characteristic polynomials split over $F$ and have zero as their only root.
\item $\mathcal{G}_{q}(F)$ denotes the subalgebra $M_{q}(F_{\mathcal{S^{\ast}}})[\mathcal{T}]$ of $M_{q}(\mathcal{C}_{F})$.
\item $\mathcal{L}_{q}(F)$ denotes the subalgebra $M_{q}(F_{\mathcal{S}})[\mathcal{T}]$ of $M_{q}(\mathcal{C}_{F})$.
\end{itemize}
It is well known from the general theory of linear recurrence sequences that
\begin{itemize}
\item $\mathcal{S}$, $\mathcal{S}^{\ast}$ and $\mathcal{S}^{\circ}$ are, respectively, bases of $F_{\mathcal{S}}$, $F_{\mathcal{S}^{\ast}}$ and $F_{\mathcal{S}^{\circ}}$.
\item $F_{\mathcal{S}}$, $F_{\mathcal{S}^{\ast}}$ and $F_{\mathcal{S}^{\circ}}$ are subalgebras of $\mathcal{C}_{F}$. More precisely, the set $\mathcal{S}^{\ast}$ is a group, and hence $F_{\mathcal{S}^{\ast}}$ is exactly the group algebra of $\mathcal{S}^{\ast}$ over $F$.
\item $F_{\mathcal{S}^{\circ}}$ and $F_{\mathcal{S}^{\ast}}$ are supplementary vector spaces relative to $F_{\mathcal{S}}$, i.e., $F_{\mathcal{S}}=F_{\mathcal{S}^{\circ}}\oplus F_{\mathcal{S}^{\ast}}$.
\end{itemize}
\section{Canonical forms for linear recurrence sequences}
\label{sect:Canonical forms}
Let $\mathcal{T}=\{\pmb{\Lambda}_{i}/i\geqslant 0\}$ be the set of all sequences $\pmb{\Lambda}_{i}=(\binom{k}{i})_{k\geq 0}$ and consider the $F$-vector space $F\langle \mathcal{T} \rangle$ spanned by $\mathcal{T}$. Then we have the following result
\begin{prop}\label{Prop 11}
Let $F$ be a field. Let $\widetilde{\mathcal{T}}$ be the set of all sequences $\widetilde{\pmb{\Lambda}}_{i}=(\binom{-k}{i})_{k\geq 0}$ and consider the sequence $\pmb{\Gamma}=(0,1,2,\ldots)$. Then we have
\begin{enumerate}[1.]
\item $F\langle \mathcal{T} \rangle$ is a subalgebra of $\mathcal{C}_{F}$ and $\mathcal{T}$ is a basis of $F\langle \mathcal{T} \rangle$.
\item $\widetilde{\mathcal{T}}$ is a basis of $F\langle \mathcal{T} \rangle$, and the linear automorphism $\chi$ of $F\langle \mathcal{T} \rangle$ which maps $\pmb{\Lambda}_{i}$ to $\widetilde{\pmb{\Lambda}}_{i}$ is an involution of the algebra $F\langle \mathcal{T} \rangle$.
\item $F\langle \mathcal{T} \rangle$ and $F_{\mathcal{S}^{\ast}}$ are $F$-linearly disjoint.
\item If $F$ is of characteristic $0$, then the sequence $\pmb{\Gamma}$ is transcendental over the ring $F_{\mathcal{S}^{\ast}}$.
\item $F_{\mathcal{S}}[\mathcal{T}]=F_{\mathcal{S}^{\circ}}\oplus F_{\mathcal{S}^{\ast}}[\mathcal{T}]$.
\end{enumerate}
\end{prop}
\begin{proof}~
\begin{enumerate}[1.]
\item Using the following formula (due to Riordan~\cite{Riord})
\begin{align*}
\binom{k}{i}\binom{k}{j}&=\sum_{m=i}^{i+j} \binom{m}{i}\binom{i}{m-j}\binom{k}{m}\\
&=\sum_{m=0}^{i} \binom{i+j-m}{m,i-m,j-m}\binom{k}{i+j-m}
\end{align*}
we can easily deduce that $\pmb{\Lambda}_{i}\pmb{\Lambda}_{j}$ is a linear combination of elements of $\mathcal{T}$ for all $i,j$. Then $F\langle \mathcal{T} \rangle$ is a subalgebra of $\mathcal{C}_{F}$. More precisely, $F\langle \mathcal{T} \rangle$ is the set of all linear recurrence sequences over $F$ with charateristic polynomials split over $F$ and have one as their only root. It is well known that $\mathcal{T}$ is a basis of $F\langle \mathcal{T} \rangle$.
\item Let $i\in \mathbb{N}$. It is clear that $\widetilde{\pmb{\Lambda}}_{0}=\pmb{\Lambda}_{0}$ and then $\widetilde{\pmb{\Lambda}}_{0}\in F\langle \mathcal{T} \rangle$. Suppose that $i\geqslant 1$. Then we have $\widetilde{\pmb{\Lambda}}_{i}=(-1)^{i}(\binom{k+i-1}{i})$. An easy application of the Vandermonde convolution shows that
$$\widetilde{\pmb{\Lambda}}_{i}=\sum_{j=1}^{i}(-1)^{i}\binom{i-1}{i-j}\pmb{\Lambda}_{j}.$$
Then $\widetilde{\pmb{\Lambda}}_{i}\in F\langle \mathcal{T} \rangle$ for all $i\in \mathbb{N}$. Since the coordinate matrix of the family $(\widetilde{\pmb{\Lambda}}_{0},\ldots,\widetilde{\pmb{\Lambda}}_{i})$ relative to $(\pmb{\Lambda}_{0},\ldots,\pmb{\Lambda}_{i})$ is the involutory lower triangular Pascal matrix of order $i+1$, it follows that $\widetilde{\mathcal{T}}$ is a basis of $F\langle \mathcal{T} \rangle$ and that $\chi$ is an involution.
\\We note that $\chi$ is an $F$-algebra automorphism is due to the fact that Riordan's formula remains true for any negative integer $k$.
\item
Let $\pmb{\lambda}_{1},\ldots,\pmb{\lambda}_{n}$ be any family of pairwise distinct elements of $\mathcal{S}^{\ast}$ and $m$ be a positive integer, and set $P(X)=\displaystyle\prod_{j=1}^{n}(X-\lambda_{j})^{m+1}$. Theorem~$1$ and Theorem~$2$ of~\cite{Fill} assure that $\{\pmb{\Lambda}_{i}\pmb{\lambda}_{j}/1\leqslant j\leqslant n \,\,\text{and}\,\, 0\leqslant i\leqslant m\}$ is a basis of the $F$-vector space of all linear recurrence sequences with characteristic polynomial $P$. It follows that the set $\{\pmb{\Lambda}_{i}\pmb{\lambda}/i\in \mathbb{N}, \pmb{\lambda}\in \mathcal{S}^{\ast}\}$ is linearly independent over the field $F$. The result now follows from Proposition~$11.6.1.$ of~\cite{Cohn} and the fact that $\mathcal{S}^{\ast}$ and $\mathcal{T}$ are, respectively, bases of the $F$-vector spaces $F_{\mathcal{S}^{\ast}}$ and $F\langle \mathcal{T} \rangle$.
\item It is well known that in the theory of linear recurrence sequences over a field $F$ with characteristic $0$, the family $\{\pmb{\Gamma}^{i}/ i\in \mathbb{N}\}$) plays the same algebraic role as that played by $\mathcal{T}$. By the same argument as in the proof of~$3.$, it turns that the $F$-vector spaces $F_{\mathcal{S}^{\ast}}$ and $F[\pmb{\Gamma}]$ are linearly disjoint. It then follows that $\pmb{\Gamma}$ is transcendental over the ring $F_{\mathcal{S}^{\ast}}$.
\item First note that $F_{\mathcal{S}^{\circ}}[\mathcal{T}]=F_{\mathcal{S}^{\circ}}$ since $\pmb{0}_{i}\pmb{\Lambda}_{j}=\binom{i}{j}\pmb{0}_{i}$ for all nonnegative integers $i$ and $j$. Now, since $$\{\pmb{0}_{0},\ldots,\pmb{0}_{n'}\}\cup\{\pmb{\Lambda}_{i}\pmb{\lambda}_{j}/\pmb{\lambda}_{0},\ldots,\pmb{\lambda}_{n}\in \mathcal{S}^{\ast} \,\,\text{and}\,\, 0\leqslant i\leqslant m\}$$ is linearly independent over the field $F$, because it is a basis of the $F$-vector space of all linear recurrence sequences with characteristic polynomial $X^{n'}\displaystyle\prod_{j=1}^{n}(X-\lambda_{j})^{m+1}$, the result $F_{\mathcal{S}}[\mathcal{T}]=F_{\mathcal{S}^{\circ}}\oplus F_{\mathcal{S}^{\ast}}[\mathcal{T}]$ follows.
\end{enumerate}
\end{proof}
\begin{rmk} To be more precise, Proposition~$11.6.1.$ of~\cite{Cohn} shows that if $A$ and $B$ are $F$-subalgebras of $\Omega$, then the following statement are equivalent:
\begin{enumerate}[1.]
\item $A$ and $B$ are $F$-linearly disjoint.
\item $\{u_{i}v_{j}\}_{i,j}$ is linearly independent over $F$ whenever $\{u_{i}\}_{i}$ is a $F$-basis of $A$ and $\{v_{j}\}_{j}$ is a $F$-basis of $B$.
\end{enumerate}
However, using this result it is easy to prove that the two following statements are also equivalent:
\begin{enumerate}[1.]
\item $A$ and $B$ are $F$-linearly disjoint.
\item There exist a $F$-basis $\{u_{i}\}_{i}$ of $A$ and a $F$-basis $\{v_{j}\}_{j}$ of $B$ such that $\{u_{i}v_{j}\}_{i,j}$ is linearly independent over $F$.
\end{enumerate}
\end{rmk}
From Proposition~\ref{Prop 11}, it follows that each $\pmb{u}\in F_{\mathcal{S}}[\mathcal{T}]$ can be written in exactly one way as $\pmb{u}=\pmb{v}+\pmb{u}_{0}\pmb{\Lambda}_{0} + \cdots + \pmb{u}_{m}\pmb{\Lambda}_{m}$, in which $\pmb{v}\in F_{\mathcal{S}^{\circ}}$ and $\pmb{u}_{0},\ldots,\pmb{u}_{m}\in F_{\mathcal{S}^{\ast}}$ and that this representation is canonical. For ease of reference, let us call
\begin{itemize}
\item $\pmb{v}+\pmb{u}_{0}\pmb{\Lambda}_{0} + \cdots + \pmb{u}_{m}\pmb{\Lambda}_{m}$ the canonical form of $\pmb{u}$ relative to $(\mathcal{S}^{\ast}, \mathcal{T})$.
\item $\pmb{u}_{0}\pmb{\Lambda}_{0} + \cdots + \pmb{u}_{m}\pmb{\Lambda}_{m}$ the geometric part of $\pmb{u}$ relative to $(\mathcal{S}^{\ast}, \mathcal{T})$.
\item $\pmb{v}$ the non-geometric part of $\pmb{u}$.
\end{itemize}
Before going further let us state the following definition.
\begin{dfn} We say a sequence of $F_{\mathcal{S}}[\mathcal{T}]$ to be purely geometric if its non-geometric part is identically zero.
\end{dfn}
\begin{lem}\label{lem 33} Let $\pmb{u}$ be a sequence over $F$. Suppose that there exist a purely geometric sequence $\pmb{w}=\pmb{w}_{0}\pmb{\Lambda}_{0} + \cdots + \pmb{w}_{m}\pmb{\Lambda}_{m}$ and a nonnegative integer $\tau$ such that $\pmb{u}(k)=\pmb{w}(k)$ for all $k\geq \tau$. Then $\pmb{u}\in F_{\mathcal{S}}[\mathcal{T}]$; in this case, $\pmb{w}_{0}\pmb{\Lambda}_{0} + \cdots + \pmb{w}_{m}\pmb{\Lambda}_{m}$ and $(\pmb{u}(0)-\pmb{w}(0))\pmb{0}_{0}+\cdots+(\pmb{u}(\tau-1)-\pmb{w}(\tau-1))\pmb{0}_{\tau-1}$ are, respectively, the geometric part and the non-geometric part of $\pmb{u}$.
\end{lem}
\begin{proof} Follows immediately from the fact that $$\pmb{u}=(\pmb{u}(0)-\pmb{w}(0))\pmb{0}_{0}+\cdots+(\pmb{u}(\tau-1)-\pmb{w}(\tau-1))\pmb{0}_{\tau-1}+\pmb{w}.$$
\end{proof}
\begin{rmk}
It is easily seen that two sequences $\pmb{u}, \pmb{v}\in F_{\mathcal{S}}[\mathcal{T}]$ are shift equivalent if and only if they have the same geometric parts.
\end{rmk}
\section{$\mathcal{P}$-canonical forms of matrices}
\label{sect:canonical forms of matrices}
Let $M_{q}(\mathcal{C}_{F})$ be the set of all matrices of order $q$ over $\mathcal{C}_{F}$ and consider the subalgebra $\mathcal{L}_{q}(F)$ of $M_{q}(\mathcal{C}_{F})$. It is straightforward to check that $\mathcal{L}_{q}(F)$ is the set of all sequences of matrices of order $q$ over $\mathcal{C}_{F}$ that are linear recurrence sequences with characteristic polynomials split over $F$, and that $M_{q}(F_{\mathcal{S}^{\circ}})$ is the subset of $\mathcal{L}_{q}(F)$ consisting of all sequences of matrices whose terms vanish from some point onwards. Since
$$F_{\mathcal{S}}[\mathcal{T}]=F_{\mathcal{S}^{\circ}}\oplus F_{\mathcal{S}^{\ast}}[\mathcal{T}],$$
it follows that
$$\mathcal{L}_{q}(F)=
M_{q}(F_{\mathcal{S}^{\circ}})\oplus \mathcal{G}_{q}(F),$$
which shows that each $\pmb{U}\in \mathcal{L}_{q}(F)$ can be written in exactly one way as $\pmb{U}=\pmb{V}+\pmb{W}$, in which $\pmb{V}\in M_{q}(F_{\mathcal{S}^{\circ}})$ and $\pmb{W}\in \mathcal{G}_{q}(F)$. Let $\pmb{U_{ij}}$ be the $(i, j)$-th entry of $\pmb{U}$, then $\pmb{V_{ij}}$ and $\pmb{W_{ij}}$ are, respectively, the non-geometric part and the geometric part of $\pmb{U_{ij}}$. We say then that $\pmb{V}$ and $\pmb{W}$ are, respectively, the non-geometric part and the geometric part of $\pmb{U}$. We note that there exist matrices $\mathcal{V}_{0},\ldots, \mathcal{V}_{n}$ with coefficients in $F$ and $\pmb{\mathcal{W}}_{0},\ldots, \pmb{\mathcal{W}}_{m}$ with coefficients in $F_{\mathcal{S}^{\ast}}$ such that
$$\pmb{V}=\mathcal{V}_{0}\pmb{0}_{0}+\cdots+\mathcal{V}_{n}\pmb{0}_{n}\quad\text{is the non-geometric part of}\quad \pmb{U}$$
and
$$\pmb{W}=\pmb{\mathcal{W}}_{0}\pmb{\Lambda}_{0}+\cdots+\pmb{\mathcal{W}}_{m}\pmb{\Lambda}_{m}\quad\text{is the geometric part of}\quad \pmb{U}.$$
We also note that, in view of Proposition~\ref{Prop 11}, the matrices $\mathcal{V}_{0},\ldots, \mathcal{V}_{n}$ and $\pmb{\mathcal{W}}_{0},\ldots, \pmb{\mathcal{W}}_{m}$ are uniquely determined by $\pmb{U}$. We conclude that every matrix $\pmb{U}$ of $\mathcal{L}_{q}(F)$ can be written in the form $$\pmb{U}=\mathcal{V}_{0}\pmb{0}_{0}+\cdots+\mathcal{V}_{n}\pmb{0}_{n}+\pmb{\mathcal{W}}_{0}\pmb{\Lambda}_{0}+
\cdots+\pmb{\mathcal{W}}_{m}\pmb{\Lambda}_{m},$$
where $\mathcal{V}_{0},\ldots, \mathcal{V}_{n}\in M_{q}(F)$ and $\pmb{\mathcal{W}}_{0},\ldots, \pmb{\mathcal{W}}_{m}\in M_{q}(F_{\mathcal{S}^{\ast}})$, and that these matrices are uniquely determined by $\pmb{U}$. Let us call this representation the $\mathcal{P}$-canonical form of $\pmb{U}$ relative to $(\mathcal{S}^{\ast}, \mathcal{T})$, and let us abbreviate this with \enquote{the $\mathcal{P}$-cf} of $\pmb{U}$ relative to $(\mathcal{S}^{\ast}, \mathcal{T})$.
\\Let now $A$ be a matrix over $F$ and let $P_{A}$ be its characteristic polynomial. It is well known that the entries of $\pmb{A}=(A^{k})_{k\geq 0}$ are linear recurrence sequences with characteristic polynomial $P_{A}$. From this it follows that if $P_{A}$ splits over $F$, then
$$\pmb{A}\in \mathcal{L}_{q}(F)=
M_{q}(F_{\mathcal{S}^{\circ}})\oplus \mathcal{G}_{q}(F),$$
and hence from the discussion above , there exist matrices $\pmb{\mathcal{A}}_{0},\ldots, \pmb{\mathcal{A}}_{l}$ with coefficients in $F_{\mathcal{S}^{\ast}}$ and a matrix $\pmb{N}(A)$ with coefficients in $F_{\mathcal{S}^{\circ}}$ such that
 $$\pmb{A}= \pmb{N}(A) + \pmb{\mathcal{A}}_{0}\pmb{\Lambda}_{0} + \cdots + \pmb{\mathcal{A}}_{l}\pmb{\Lambda}_{l},$$
and this representation is the $\mathcal{P}$-cf of $\pmb{A}$ (or simply of $A$) relative to $(\mathcal{S}^{\ast}, \mathcal{T})$.
\\Obviously, when $F$ is a field of characteristic $0$, we obtain the same results as above if we decide to use $\mathcal{H}$ instead of $\mathcal{T}$. In this case the $\mathcal{P}$-cf of $A$ relative to $(\mathcal{S}^{\ast}, \mathcal{H})$ has the form
$$\pmb{A}= \pmb{N}(A) + \pmb{\mathcal{A'}}_{0}\pmb{\Gamma}^{0} + \cdots + \pmb{\mathcal{A'}}_{l}\pmb{\Gamma}^{l}$$
For simplicity, in the sequel we call $\pmb{\mathcal{A}}_{i}$ and $\pmb{\mathcal{A'}}_{i}$ the $i$th coordinate of $A$ with respect to $(\mathcal{S}^{\ast}, \mathcal{T})$ and $(\mathcal{S}^{\ast}, \mathcal{H})$ respectively.
\\In the following proposition we give some properties of the matrices $\pmb{N}(A),\pmb{\mathcal{A}}_{0},\ldots, \pmb{\mathcal{A}}_{l}$.
\begin{prop}\label{Prop 23} Let $A$ be a square matrix of order $q$ over $F$, and suppose that its minimal polynomial is of the form $m_{A}=X^{t_{0}}\prod_{j=1}^{p}\displaystyle(X-\lambda_{j})^{t_{j}}$ with the $\lambda_{j}$ distinct and belong to $F$(possibly $t_{0}=0$). Let $m=\max\{t_{1}-1,\ldots,t_{p}-1\}$ and let $\pi_{0},\ldots,\pi_{p}$ be the spectral projections of $A$ at $0,\lambda_{1},\ldots,\lambda_{p}$, respectively. Then we have
\begin{enumerate}[1.]
\item $\pmb{\mathcal{A}}_{i}=\displaystyle\sum_{j=1}^{p}\pmb{\lambda}_{j}\lambda_{j}^{-i}(A-\lambda_{j}I_{q})^{i}\pi_{j}$, $i\geqslant 1$, and $\pmb{\mathcal{A}}_{0}=\pmb{\lambda}_{1}\pi_{1}+\cdots+\pmb{\lambda}_{p}\pi_{p}$.
\item $\pmb{N}(A)=\displaystyle\sum_{i=0}^{t_{0}-1}\pmb{0}_{i}A^{i}\pi_{0}$ ($\pmb{N}(A)=0$ if and only if $t_{0}=0$).
\item $A$ is nilpotent if and only if each $\pmb{\mathcal{A}}_{i}$ is zero, i.e., $\pmb{A}= \pmb{N}(A)$.
\item $\pmb{\mathcal{A}}_{i}=0$ if and only if $i> m$.
\item If $A$ is not nilpotent, then no coordinate $\pmb{\mathcal{A}}_{i}$, $1\leq i\leq m$, is identically zero.
\item $\pmb{\mathcal{A}}_{i}=\displaystyle\sum_{j=1}^{p}\pmb{\lambda}_{j}\lambda_{j}^{-i}(A-\lambda_{j}I_{q})^{i}\pi_{j}$ is a nilpotent matrix, $i\geqslant 1$.
\item $\pmb{N}(A)=\pmb{0}_{0}\pi_{0}+\pmb{N}$ where $\pmb{N}=\displaystyle\sum_{i=1}^{t_{0}-1}\pmb{0}_{i}A^{i}\pi_{0}$ is a nilpotent matrix.
\item $\pmb{A}= (\pmb{N} + \pmb{\mathcal{A}}_{1}\pmb{\Lambda}_{1} + \cdots + \pmb{\mathcal{A}}_{m}\pmb{\Lambda}_{m}) + (\pmb{0}_{0}\pi_{0}+\pmb{\mathcal{A}}_{0}\pmb{\Lambda}_{0})$ is the unique decomposition of $\pmb{A}$ as a sum of commuting nilpotent and diagonalizable matrices.
\end{enumerate}
\end{prop}
\begin{proof}~
 \\1. and 2. We know that $$I_{q}=\pi_{0}+\pi_{1}+\cdots+\pi_{p},$$ then $$A=(A\pi_{0})\pi_{0}+(A\pi_{1})\pi_{1}+\cdots+(A\pi_{p})\pi_{p}$$ and hence $$A^{k}=(A\pi_{0})^{k}\pi_{0}+(A\pi_{1})^{k}\pi_{1}+\cdots+(A\pi_{p})^{k}\pi_{p}$$ for all $k\geqslant 0$. But since $A\pi_{i}=(A-\lambda_{i}I_{q})\pi_{i}+\lambda_{i}\pi_{i}$ and $(A-\lambda_{i}I_{q})^{t_{i}}\pi_{i}=0$, we see that $$A^{k}=\sum_{i=0}^{t_{0}-1}0^{k-i}\binom{k}{i}A^{i}\pi_{0}+\sum_{i=1}^{p}(\sum_{j=0}^{s}\lambda_{i}^{k-j}(A-\lambda_{i}I_{q})^{j}\pi_{i}\binom{k}{j})$$
for all $k\geqslant 0$ and $s\geqslant \max\{t_{1}-1,\ldots,t_{p}-1\}$.
Therefore, $$\pmb{A}= \pmb{N}(A) + \pmb{\mathcal{A}}_{0}\pmb{\Lambda}_{0} + \cdots + \pmb{\mathcal{A}}_{s}\pmb{\Lambda}_{s},$$ where $$\pmb{\mathcal{A}}_{i}=\displaystyle\sum_{j=1}^{p}\pmb{\lambda}_{j}\lambda_{j}^{-i}(A-\lambda_{j}I_{q})^{i}\pi_{j}, 1\leqslant i\leqslant s,$$
$$\pmb{\mathcal{A}}_{0}=\pmb{\lambda}_{1}\pi_{1}+\cdots+\pmb{\lambda}_{p}\pi_{p}$$ and $$\pmb{N}(A)=\displaystyle\sum_{i=0}^{t_{0}-1}\pmb{0}_{i}A^{i}\pi_{0}.$$
\\3. Is evident.
\\4. Suppose $\pmb{\mathcal{A}}_{i}=0$. Then we have $\displaystyle\sum_{j=1}^{p}\pmb{\lambda}_{j}\lambda_{j}^{-i}(A-\lambda_{j}I_{q})^{i}\pi_{j}=0$. Multiplying by $\pi_{r}$, we get $\pmb{\lambda}_{r}\lambda_{r}^{-i}(A-\lambda_{r}I_{q})^{i}\pi_{r}=0$ and hence $(A-\lambda_{r}I_{q})^{i}\pi_{r}=0$. The result now follows from the well-known fact that $t_{r}$ is the smallest nonnegative integer satisfying $$(A-\lambda_{r}I_{q})^{t_{r}}\pi_{r}=0.$$
\\5. Is evident.
\\6. Follows from the fact that if $i\geqslant 1$ then $\pmb{\mathcal{A}}_{i}$ is a sum of commuting nilpotent matrices.
\\7. Is evident.
\\8. Let $A=A_{n}+A_{s}$ be the Jordan-Chevalley decomposition of $A$. Then $A^{k}=(A_{n}+A_{s})^{k}=B_{k}+A^{k}_{s}$ where $B_{k}$ is a matrix which commutes with $A_{n}$ and $A_{s}$. More precisely, $A^{k}=B_{k}+A^{k}_{s}$ is the Jordan-Chevalley decomposition of $A^{k}$. Since $(A^{k}_{s})_{k}=\pmb{0}_{0}\pi_{0}+\pmb{\mathcal{A}}_{0}\pmb{\Lambda}_{0}$ it follows that $(B_{k})_{k}=\pmb{N} + \pmb{\mathcal{A}}_{1}\pmb{\Lambda}_{1} + \cdots + \pmb{\mathcal{A}}_{m}\pmb{\Lambda}_{m}$. The fact that the matrices $A_{s}^{k}, k\geqslant 0$ are simultaneously diagonalizable guarantees that the sequence of matrices $(A_{s}^{k})_{k}$ is diagonalizable. The unicity of the decomposition follows from the fact that, for every $k\in \mathbb{N}$, the matrix $A^{k}$ can be uniquely expressed as the sum of commuting diagonalizable and nilpotent matrices.
\end{proof}
\begin{rmk}
In the above (and in the rest of this paper) we have identified sequences of matrices and their associated matrices of sequences.
\end{rmk}
\begin{rmk}
$A$ is invertible if and only if $\pmb{N}(A)=0$.
\end{rmk}
\begin{rmk}
If we plug $k=0$ into $(\pmb{\mathcal{A}}_{0}\pmb{\Lambda}_{0} + \cdots + \pmb{\mathcal{A}}_{m}\pmb{\Lambda}_{m})_{k}$, we get $\pi_{1}+\cdots+\pi_{p}=I_{q}-\pi_{0}$. Thus, if we have already determined the geometric part of a matrix $A$, then we can easily check whether or not $A$ is nonsingular. For this, it is sufficient to plug $k=0$ into the geometric part of $A$ and see if we get the identity matrix or not.
\end{rmk}
We have the following corollaries
\begin{cor}\label{Cor 21} Let $A$ be a square matrix over $F$ and let $t_{0}$ be its index. Suppose there exists a purely geometric sequence $A(k)_{k\geq 0}$ and a nonnegative integer $\tau$ such that $A^{k}=A(k)$ for all $k\geq \tau$. Then we have
\\1. $t_{0}\leq \tau$.
\\2. $t_{0}$ is the smallest nonnegative integer for which $A^{t_{0}}=A(t_{0})$. Furthermore, we have $A^{k}=A(k)$ for all $k\geq t_{0}$.
\end{cor}
\proof
On the one hand, according to Lemma~\ref{lem 33}, $$\pmb{N}(A)=(I-A(0))\pmb{0}_{0}+(A-A(1))\pmb{0}_{1}+\cdots+(A^{\tau-1}-A(\tau-1))\pmb{0}_{\tau-1}.$$
On the other hand, according to Proposition~\ref{Prop 23},
$$\pmb{N}(A)=\displaystyle\sum_{i=0}^{t_{0}-1}\pmb{0}_{i}A^{i}\pi_{0}.$$
Therefore, it follows that $t_{0}\leq \tau$ Since $t_{0}$ is the smallest nonnegative integer for which $A^{t_{0}}\pi_{0}=0$. Furthermore, the unicity of the representation of $\pmb{N}(A)$ ensures that $A^{i}-A(i)=A^{i}\pi_{0}\neq 0$, $0\leq i\leq t_{0}-1$ and $A^{i}-A(i)=0$, $t_{0}\leq i\leq \tau$. Therefore,  $t_{0}$ is the smallest nonnegative integer for which $A^{t_{0}}=A(t_{0})$. The last assertion of (2) is also immediate
after the previous discussion. This completes the proof of the corollary.
\endproof
\begin{cor}\label{Cor 22} Let $A$ be a square matrix over $F$. Then the minimal polynomial of $A$ is $m_{A}(X)=X^{t_{0}}\prod_{j=1}^{p}(X-\lambda_{j})^{t_{j}}$, where $t_{0}$ and $t_{j}, j\neq 0$, are, respectively, the greatest integers such that $\pmb{0}_{t_{0}-1}$ and $\pmb{\lambda}_{j}\pmb{\Lambda}_{t_{j}-1}$ appear in the $\mathcal{P}$-cf of $A$ relative to $(\mathcal{S}^{\ast}, \mathcal{T})$.
\end{cor}
\proof
The proof of the Corollary follows immediately from the fact that $t_{0}$ and $t_{j}$, $1\leq j \leq p$ are the the smallest nonnegative integer for which $A^{t_{0}}\pi_{0}=0$ and $(A-\lambda_{j}I_{q})^{t_{j}}\pi_{j}=0$.
\endproof
\begin{cor}\label{Cor 23} Let $A$ be a square matrix over $F$ such that its characteristic polynomial splits over $F$. Then $A$ is diagonalizable if and only if the geometric part of $A$ is an element of $M_{q}(F_{\mathcal{S}^{\ast}})$ and the non-geometric part of $A$ is an element of $\pmb{0}_{0}M_{q}(F)$.
\end{cor}

The Jordan-Chevalley decomposition of a matrix $A$ can certainly be used to determine the $\mathcal{P}$-cf of $A$ relative to $(\mathcal{S}^{\ast}, \mathcal{T})$. The corollary below shows that the converse is also true.
\begin{cor}\label{Cor 11}
Let $A$ be a square matrix over $F$. Then the diagonalizable part of $A$ is $A_{s}=\pmb{\mathcal{A}}_{0}(1)$ and the nilpotent part of $A$ is $A_{n}=A-A_{s}=\pmb{N}(A)(1)+\pmb{\mathcal{A}}_{1}(1)$, where $\pmb{\mathcal{A}}_{0}$ and $\pmb{\mathcal{A}}_{1}$ are the $0$th and the $1$th coordinates of $A$ with respect to $(\mathcal{S}^{\ast}, \mathcal{T})$.
\end{cor}
\proof Follows immediately from Property~$8$ of Proposition~\ref{Prop 23}.
\endproof

As an illustration, consider the following example taken from~\cite{Male}.
\begin{ex}\label{example 1}
Let $$A=\begin{pmatrix}
1&1&1&0\\1&1&1&-1\\0&0&-1&1\\0&0&1&-1
\end{pmatrix}$$
and let $A(k)$ be the sequence of matrices
$$A(k)=\begin{pmatrix}
2^{-1+k}&2^{-1+k}&\frac{1}{16}2^{k}((-1)^{1+k}+5)&\frac{1}{16}2^{k}((-1)^{k}-1)\vspace*{0.1pc}\\
2^{-1+k}&2^{-1+k}&\frac{5}{16}2^{k}((-1)^{1+k}+1)&\frac{1}{16}2^{k}(5(-1)^{k}-1)\vspace*{0.1pc}\\
0&0&(-1)^{k}2^{-1+k}&(-1)^{1+k}2^{-1+k}\vspace*{0.1pc}\\
0&0&(-1)^{1+k}2^{-1+k}&(-1)^{k}2^{-1+k}
\end{pmatrix}.$$
It is proven in~\cite{Male} that $A^{k}=A(k)$ for all $k\geq 4$.
\\We have $A(0)\neq I_{4}$, then $A$ is singular.
\\We also have $A(1)\neq A$, $A(2)=A^{2}$. Then $A^{k}=A(k)$ for all $k\geq 2$.
\\From these results, we may conclude that
\begin{itemize}
\item The non-geometric part of $A$ is $\pmb{N}(A)=(I_{4}-A(0))\pmb{0}_{0}+(A-A(1))\pmb{0}_{1}$.
\item The geometric part of $A$ is
$$\begin{pmatrix}
2^{-1}(2^{k})&2^{-1}(2^{k})&\frac{5}{16}(2^{k})-\frac{1}{16}((-2)^{k})&\frac{-1}{16}(2^{k})+\frac{1}{16}((-2)^{k})\vspace*{0.1pc}\\
2^{-1}(2^{k})&2^{-1}(2^{k})&\frac{5}{16}(2^{k})-\frac{5}{16}((-2)^{k})&\frac{-1}{16}(2^{k})+\frac{5}{16}((-2)^{k})\vspace*{0.1pc}\\
0&0&2^{-1}((-2)^{k})&-2^{-1}((-2)^{k})\vspace*{0.1pc}\\
0&0&-2^{-1}((-2)^{k})&2^{-1}((-2)^{k})
\end{pmatrix}$$
\item The minimal polynomial of $A$ is $X^{2}(X-2)(X+2)$.
\item $A_{s}=\pmb{\mathcal{A}}_{0}(1)=A(1)$ and $A_{n}=\pmb{N}(A)(1)=A-A(1)$
\item
\begin{equation*}
\pi_{0}=I_{4}-A(0)=\begin{pmatrix}
2^{-1}&-2^{-1}&-\frac{1}{4}&0\vspace*{0.1pc}\\
-2^{-1}&2^{-1}&0&\frac{-1}{4}\vspace*{0.1pc}\\
0&0&2^{-1}&2^{-1}\vspace*{0.1pc}\\
0&0&2^{-1}&2^{-1}
\end{pmatrix},\,
\pi_{2}=\begin{pmatrix}
2^{-1}&2^{-1}&\frac{5}{16}&\frac{-1}{16}\vspace*{0.1pc}\\
2^{-1}&2^{-1}&\frac{5}{16}&\frac{-1}{16}\vspace*{0.1pc}\\
0&0&0&0\vspace*{0.1pc}\\
0&0&0&0
\end{pmatrix},
\end{equation*}
\begin{equation*}
\pi_{-2}=\begin{pmatrix}
0&0&\frac{-1}{16}&\frac{1}{16}\vspace*{0.1pc}\\
0&0&\frac{-5}{16}&\frac{5}{16}\vspace*{0.1pc}\\
0&0&2^{-1}&-2^{-1}\vspace*{0.1pc}\\
0&0&-2^{-1}&2^{-1}
\end{pmatrix}
\end{equation*}
\end{itemize}
\end{ex}
\section{$\mathcal{P}$-canonical forms of the Drazin inverses of matrices}
Consider the following linear map
\begin{align*}
\theta: F_{\mathcal{S}^{\ast}} &\longrightarrow F_{\mathcal{S}^{\ast}}\\
\pmb{\lambda} &\longmapsto \pmb{\lambda}^{-1}
\end{align*}
Obviously, $\theta(\mathcal{S}^{\ast})=\mathcal{S}^{\ast}$ and $\theta^{-1}=\theta$. Since $\theta(\pmb{\lambda}\pmb{\mu})=\theta(\pmb{\lambda})\theta(\pmb{\mu})$ for all $\pmb{\lambda}, \pmb{\mu}\in \mathcal{S}^{\ast}$, it follows that $\theta$ is an $F$-algebra automorphism of $F_{\mathcal{S}^{\ast}}$.
\\Let \begin{align*}
\Psi: F_{\mathcal{S}^{\circ}} \longrightarrow F_{\mathcal{S}^{\circ}}
\end{align*}
be any $F$-endomorphism of $F_{\mathcal{S}^{\circ}}$ and consider
\begin{align*}
\theta_{\Psi}=\Psi\oplus \theta: F_{\mathcal{S}} \longrightarrow F_{\mathcal{S}}
\end{align*}
The direct sum of $\theta$ and $\Psi$. It is obvious that $\theta_{\Psi}$ is an $F$-endomorphism of $F_{\mathcal{S}}$, involutory if $\Psi$ is, but it is not necessarily an algebra homomorphism; in fact, we have $\pmb{\lambda}\pmb{0}_{i}=\lambda^{i}\pmb{0}_{i}$, then $\theta_{\Psi}(\pmb{\lambda}\pmb{0}_{i})=\lambda^{i}\Psi(\pmb{0}_{i})$ which is not always equal to $\theta_{\Psi}(\pmb{\lambda})\theta_{\Psi}(\pmb{0}_{i})=\pmb{\lambda}^{-1}\Psi(\pmb{0}_{i})$.
\\Let us denote by $\widetilde{\Psi}$ the $F$-endomorphism
\begin{align*}
\widetilde{\Psi}: M_{q}(F_{\mathcal{S}^{\circ}}) &\longrightarrow M_{q}(F_{\mathcal{S}^{\circ}})\\
(A_{ij}) &\longmapsto (\Psi(A_{ij}))
\end{align*}
\\Let us also denote by $\widetilde{\theta}$ the $F$-algebra automorphism
\begin{align*}
\widetilde{\theta}: M_{q}(F_{\mathcal{S}^{\ast}}) &\longrightarrow M_{q}(F_{\mathcal{S}^{\ast}})\\
(A_{ij}) &\longmapsto (\theta(A_{ij}))
\end{align*}
Note that $\widetilde{\theta}$ is an involution of $M_{q}(F_{\mathcal{S}^{\ast}})$.
\\Let $\widetilde{\theta}_{\Psi}$ denote the linear endomorphism induced by $\widetilde{\Psi}\oplus \widetilde{\theta}$ on
$$\mathcal{L}_{q}(F)=
M_{q}(F_{\mathcal{S}^{\circ}})\oplus \mathcal{G}_{q}(F)$$
defined by
\begin{align*}
\widetilde{\theta}_{\Psi}: \mathcal{L}_{q}(F) &\longrightarrow \mathcal{L}_{q}(F)\\
\pmb{U} &\longmapsto \mathcal{V}_{0}\Psi(\pmb{0}_{0})+\cdots+\mathcal{V}_{n}\Psi(\pmb{0}_{n})+
\widetilde{\theta}(\pmb{\mathcal{W}}_{0})\widetilde{\pmb{\Lambda_{0}}}+\cdots+\widetilde{\theta}(\pmb{\mathcal{W}}_{m})\widetilde{\pmb{\Lambda_{m}}},
\end{align*}
where $$\mathcal{V}_{0}\pmb{0}_{0}+\cdots+\mathcal{V}_{n}\pmb{0}_{n}+\pmb{\mathcal{W}}_{0}\pmb{\Lambda}_{0}+\cdots+\pmb{\mathcal{W}}_{m}\pmb{\Lambda}_{m}$$
is the $\mathcal{P}$-cf of $\pmb{U}$ relative to $(\mathcal{S}^{\ast}, \mathcal{T})$. The map $\widetilde{\theta}_{\Psi}$ is well-defined, since $\widetilde{\pmb{\Lambda}_{i}}\in F\langle\mathcal{T}\rangle$ for all non-negative integer $i$ (see Proposition~\ref{Prop 11}).
\\In the remainder of this section, we will only be interested in the special case when $\Psi$ is the zero map. Let us denote  the map $$0\oplus \theta: F_{\mathcal{S}} \longrightarrow F_{\mathcal{S}}$$
simply by $\theta_{0}$. It should be noted that the image of a sequence $\pmb{u}$ of $F_{\mathcal{S}}$ under $\theta_{0}$ can be obtained by simply plugging in $-k$ for $k$ in the geometric part of $\pmb{u}$ relative to $(\mathcal{S}^{\ast}, \mathcal{T})$ and neglecting its non-geometric part.
\\Let $A$ be a square matrix of index $t_{0}$ with characteristic polynomial splits over $F$. We may assume that the Jordan canonical form of $A$ has the form as follows
$$A=P\begin{pmatrix} D&0\\ 0&N \end{pmatrix}P^{-1}$$
where $P$ is a nonsingular matrix, $D$ is a nonsingular matrix of order $r=\text{rank}(A^{t_{0}})$, and $N$ is a
nilpotent matrix such that $N^{t_{0}}=0$. Then we can write the matrix $\pmb{A}$ in the form
$$\pmb{A}=P\begin{pmatrix} 0&0\\ 0&N^{k} \end{pmatrix}_{\hspace*{-0.2pc} k}P^{-1} + P\begin{pmatrix} D^{k}&0\\ 0&0 \end{pmatrix}_{\hspace*{-0.2pc} k}P^{-1}.$$
It is clear that these two matrices are the non-geometric and the geometric parts of $A$ written in a form other than the $\mathcal{P}$-cf. If we plug $-k$ for $k$ into this form of the geometric part of $A$, we get the matrix $\pmb{A}_{d}-\pi_{0}\pmb{0}_{0}$. In the theorem below, we show that the same remains true for the $\mathcal{P}$-cf of $A$ relative to $(\mathcal{S}^{\ast}, \mathcal{T})$.
\begin{thm}\label{thm 535}
Let $A$ be a square matrix of index $t_{0}$ with coefficients in $F$ such that its characteristic polynomial splits over $F$. Then we have $\pmb{A}_{d}=\widetilde{\theta}_{0}(\pmb{A})+\pmb{0}_{0}\pi_{0}$, where $\pmb{A}_{d}=(A_{d}^{k})_{k}$.
\end{thm}
\begin{proof} We must proof that
$$\begin{array}{rrlc}
\pmb{A}^{t_{0}+1}(\widetilde{\theta}_{0}(\pmb{A})+\pmb{0}_{0}\pi_{0})&=&\pmb{A}^{t_{0}}\quad\quad &(1^{t_{0}})\\
(\widetilde{\theta}_{0}(\pmb{A})+\pmb{0}_{0}\pi_{0})\pmb{A}(\widetilde{\theta}_{0}(\pmb{A})+
\pmb{0}_{0}\pi_{0})&=&\widetilde{\theta}_{0}(\pmb{A})+\pmb{0}_{0}\pi_{0} \quad\quad &(3)\\
\pmb{A}(\widetilde{\theta}_{0}(\pmb{A})+\pmb{0}_{0}\pi_{0})&=&(\widetilde{\theta}_{0}(\pmb{A})+\pmb{0}_{0}\pi_{0})\pmb{A} \quad\quad &(5)
\end{array}$$
From Proposition~\ref{Prop 23} it follows that
$$\pmb{A}=\displaystyle\sum_{i=0}^{t_{0}-1}\pmb{0}_{i}A^{i}\pi_{0} + \sum_{i=0}^{m}(\sum_{j=1}^{p}\pmb{\lambda}_{j}\lambda_{j}^{-i}(A-\lambda_{j}I_{q})^{i}\pi_{j})\pmb{\Lambda}_{i}$$ and then $$\widetilde{\theta}_{0}(\pmb{A})=
\displaystyle\sum_{i=0}^{m}(\sum_{j=1}^{p}\pmb{\lambda}^{-1}_{j}\lambda_{j}^{-i}(A-\lambda_{j}I_{q})^{i}\pi_{j})
\widetilde{\pmb{\Lambda}}_{i},$$ where $m=\max\{t_{1}-1,\ldots,t_{p}-1\}$.
\\Clearly, $\widetilde{\theta}_{0}(\pmb{A})+\pmb{0}_{0}\pi_{0}$ satisfies $(5)$.
\\In order to prove that $\widetilde{\theta}_{0}(\pmb{A})+\pmb{0}_{0}\pi_{0}$ satisfies $(3)$, observe first that $$\pmb{\mathcal{A}}_{i}\pmb{\Lambda}_{i}\widetilde{\theta}_{0}(\pmb{\mathcal{A}}_{l}\pmb{\Lambda}_{l})=
\displaystyle\sum_{j=1}^{p}\lambda_{j}^{-(i+l)}(A-\lambda_{j}I_{q})^{i+l}\pi_{j}\pmb{\Lambda}_{i}\widetilde{\pmb{\Lambda}}_{l}.$$
As $(A-\lambda_{j}I_{q})^{m+1}\pi_{j}=0$, it follows that $$\pmb{A}\widetilde{\theta}_{0}(\pmb{A})=
\displaystyle\sum_{j=1}^{p}\sum_{i=0}^{m}\lambda_{j}^{-i}(A-\lambda_{j}I_{q})^{i}\pi_{j})
(\sum_{l=0}^{i}\pmb{\Lambda}_{l}\widetilde{\pmb{\Lambda}}_{i-l}).$$
Furthermore, since $\displaystyle\sum_{l=0}^{i}\binom{k}{l}\binom{-k}{i-l}=0$ for all positive integer $i$, we also have $$\displaystyle\sum_{l=0}^{i}\pmb{\Lambda}_{l}\widetilde{\pmb{\Lambda}}_{i-l}=\pmb{0}.$$
Therefore, $$\pmb{A}\widetilde{\theta}_{0}(\pmb{A})=\displaystyle\sum_{j=1}^{p}\pi_{j}=
I_{q}-\pi_{0}.$$ Hence $$\widetilde{\theta}_{0}(\pmb{A})\pmb{A}\widetilde{\theta}_{0}(\pmb{A})=\widetilde{\theta}_{0}(\pmb{A})-
\pi_{0}\widetilde{\theta}_{0}(\pmb{A})=
\widetilde{\theta}_{0}(\pmb{A})$$ and because of $\pmb{A}\pmb{0}_{0}^{2}\pi_{0}^{2}=\pmb{0}_{0}\pi_{0}$, we conclude that
$$(\widetilde{\theta}_{0}(\pmb{A})+\pmb{0}_{0}\pi_{0})\pmb{A}(\widetilde{\theta}_{0}(\pmb{A})+
\pmb{0}_{0}\pi_{0})=\widetilde{\theta}_{0}(\pmb{A})+\pmb{0}_{0}\pi_{0}.$$
To complete the proof, it remains to show that $\widetilde{\theta}_{0}(\pmb{A})+\pmb{0}_{0}\pi_{0}$ satisfies $(1^{t_{0}})$. Since $\pmb{A}\widetilde{\theta}_{0}(\pmb{A})=I_{q}-\pi_{0}$ and $\pmb{A}^{t_{0}}\pi_{0}=\pmb{0}_{0}\pi_{0}$, it follows that $$\pmb{A}^{t_{0}+1}\widetilde{\theta}_{0}(\pmb{A})=\pmb{A}^{t_{0}}-\pmb{A}^{t_{0}}\pi_{0}=\pmb{A}^{t_{0}}-\pmb{0}_{0}\pi_{0}.$$ But since $\pmb{A}^{t_{0}+1}\pmb{0}_{0}\pi_{0}=\pmb{0}_{0}\pi_{0}$, it follows that $$\pmb{A}^{t_{0}+1}(\widetilde{\theta}_{0}(\pmb{A})+\pmb{0}_{0}\pi_{0})=\pmb{A}^{t_{0}}.$$
The uniqueness of a $\{1^{t_{0}},3,5\}$-inverse $(B_{k})_{k}$ of $\pmb{A}$ follows from the uniqueness of each $B_{k}$ which is the unique $\{1^{t_{0}},3,5\}$-inverse of $A^{k}$, for all positive integer $k$. We conclude that $\widetilde{\theta}_{0}(\pmb{A})+\pmb{0}_{0}\pi_{0}$ is the $\{1^{t_{0}},3,5\}$-inverse of $\pmb{A}$.
\end{proof}
\begin{rmk}~
\\ Although we have specified that the set $F$ is a field and that the characteristic polynomial of $A$ splits over $F$, the theorem~\ref{thm 535} remains true in any integral domain $R$ and for any square matrix. In this case, $A_{d}$ is a matrix over the algebraic closure of the quotient field of $R$.
\end{rmk}
\begin{rmk}~
\\ If $F=\mathbb{C}$ is the complex field, then
\begin{eqnarray*}
A_{d}&=&(\pmb{A}_{d})(1)=\widetilde{\theta}_{0}(\pmb{A})(1)\\
&=&\displaystyle\sum_{j=1}^{p} \pi_{j} \sum_{i=0}^{m}\lambda_{j}^{-i-1}(A-\lambda_{j}I_{q})^{i})(\widetilde{\pmb{\Lambda}}_{i})(1)\\
&=&\sum_{j=1}^{p} \pi_{j} \sum_{i=0}^{m}(-1)^{i}\lambda_{j}^{-i-1}(A-\lambda_{j}I_{q})^{i}).
\end{eqnarray*}
We find then the well-known result that $A_{d}=\displaystyle\frac{1}{z}(A)$ is the matrix function corresponding to the reciprocal $f(z)= \displaystyle\frac{1}{z}$, defined on nonzero eigenvalues (see e.g. Corollary~1. p. 165 of~\cite{Ben}).
\end{rmk}
Let now $\mathcal{U}$ be a basis of the free $F_{\mathcal{S}^{\ast}}$-module $F_{\mathcal{S}^{\ast}}\langle \mathcal{T} \rangle$. Since $\mathcal{U}$ has the same cardinality as $\mathcal{T}$, it can be indexed by $\mathbb{N}$, i.e., $\mathcal{U}=\{\pmb{\mathcal{U}}_{i}\}_{i\in \mathbb{N}}$.
\begin{thm}\label{Thm : 010}  Let $F$ be a field of characteristic $0$ and $\mathcal{U}=\{\pmb{\mathcal{U}}_{i}\}_{i\in \mathbb{N}}$ be a basis of the $F_{\mathcal{S}^{\ast}}$-module $F_{\mathcal{S}^{\ast}}\langle \mathcal{T} \rangle$. Assume that $\mathcal{U}$ is contained in the $F$-vector space spanned by $\mathcal{T}$. Then we have the following:
\begin{enumerate}[1.]
\item For all nonnegative integer $i$, there exists a polynomial $P_{i}(X)\in F[X]$ such that $\pmb{\mathcal{U}}_{i}=(P_{i}(k))_{k\geq0}$.
\item $\widetilde{\theta}(\pmb{\mathcal{U}}_{i})=(P_{i}(-k))_{k\geq0}$ For all nonnegative integer $i$.
\item Let $A$ be a square matrix over $F$ with characteristic polynomial splits over $F$. Let $$\pmb{A}=\pmb{N}(A) + \pmb{\mathcal{B}}_{0}(P_{0}(k))_{k} + \cdots + \pmb{\mathcal{B}}_{m}(P_{m}(k))_{k}$$ be the representation of $A$ with respect to $(\mathcal{S}^{\ast}, \mathcal{U})$, which is the $\mathcal{P}$-cf of $A$ relative to $(\mathcal{S}^{\ast}, \mathcal{U})$. Then the sequence of matrices obtained by plugging $-k$ for $k$ in the geometric part $$\pmb{\mathcal{B}}_{0}(P_{0}(k))_{k} + \cdots + \pmb{\mathcal{B}}_{m}(P_{m}(k))_{k}$$ of $A$ is $\pmb{A}_{d}-\pi_{0}\pmb{0}_{0}$.
\end{enumerate}
\end{thm}
\proof~
\begin{enumerate}[1.]
\item It is well known that when the characteristic of $F$ is $0$, $\binom{k}{j}$ is a polynomial in $k$, i.e., there exists $Q_{j}(X)\in F[X]$ such that $\binom{k}{j}=Q_{j}(k)$ and then $\pmb{\Lambda}_{j}=(Q_{j}(k))_{k\geq 0}$. Since $\mathcal{U}$ is contained in the $F$-vector space spanned by $\mathcal{T}$, there exist elements $a_{ij}$ in $F$, such that
$$\pmb{\mathcal{U}}_{i}=\sum_{finite} a_{ij}\pmb{\Lambda}_{j}.$$
In other words, $\pmb{\mathcal{U}}_{i}=(P_{i}(k))_{k\geq0}$, where
$$P_{i}(k)=\sum_{finite} a_{ij}Q_{j}(k).$$
\item It is easily seen that the equality $\binom{k}{j}=Q_{j}(k)$, which is between two different forms of the same sequence, still valid when we plug in $-k$ for $k$. Then
\begin{eqnarray*}
\widetilde{\theta}(\pmb{\mathcal{U}}_{i})&=&\sum_{finite} a_{ij}\widetilde{\theta}(\pmb{\Lambda}_{j})\\
&=&\sum_{finite} a_{ij}(\binom{-k}{j})_{k\geq 0}\\
&=&\sum_{finite} a_{ij}(Q_{j}(-k))_{k\geq 0}\\
&=&(P_{i}(-k))_{k\geq 0}.
\end{eqnarray*}
\item The result follows immediately from $2.$ above together with Theorem~\ref{thm 535}.
\end{enumerate}
\endproof
As a particular consequence of the above theorem, we obtain the following result.
\begin{thm} Let $F$ be a field of characteristic $0$ and $\mathcal{H}=(\pmb{\Gamma}^{i})_{i\geqslant 0}$, where $\pmb{\Gamma}$ is the sequence $\pmb{\Gamma}=(0,1,2,\ldots)$. Let $A$ be a square matrix over $F$ with characteristic polynomial splits over $F$. Then the sequence of matrices obtained by plugging $-k$ for $k$ in the geometric part of $A$ relative to $(\mathcal{S}^{\ast}, \mathcal{H})$ is $\pmb{A}_{d}-\pi_{0}\pmb{0}_{0}$.
\end{thm}
 \proof It is well known that
 \begin{equation*}
(\pmb{\Gamma}^{i})_{i\geqslant 0}=T(\pmb{\Lambda_{i}})_{i\geqslant 0}
\end{equation*}
where $T$ is the infinite invertible lower triangular matrix $(m!S(n,m))_{n,m\geqslant 0}$ and $S(n,m)$ are the Stirling numbers of the second kind (see, e.g. Quaintance and Gould~\cite{Qua}).
\\From this result, we deduce on the one hand that $\mathcal{H}$ is contained in the $F$-vector space spanned by $\mathcal{T}$ and on the other hand that $\mathcal{H}$ is, as is well known, a basis of the $F_{\mathcal{S}^{\ast}}$-module $F_{\mathcal{S}^{\ast}}\langle \mathcal{T} \rangle$. The corollary is then a consequence of Theorem~\ref{Thm : 010}.
\endproof
Now let us consider the case where $F=\mathbb{C}$ the complex field. Consider the following sets
\begin{itemize}
\item $\mathcal{S}^{\ast}=\{\pmb{\lambda}=(\lambda^{k})_{k\geqslant 0}/\lambda\in \mathbb{C}, \lambda\neq 0\}$
\item $\widetilde{S}=\{\pmb{\lambda}/\lambda\in \mathbb{R}, \lambda\neq 0\}$
\item $\mathcal{S}^{+}=\{\pmb{\lambda}\in \mathcal{S}^{\ast}/ Im(\lambda)> 0\}$
\item $\mathcal{S}^{-}=\{\pmb{\lambda}\in \mathcal{S}^{\ast}/ Im(\lambda)< 0\}=\{\pmb{\overline{\lambda}}/ \pmb{\lambda}\in \mathcal{S}^{+}\}$
\item $\mathcal{S}^{+}_{1}=\{\frac{\pmb{\lambda}+\pmb{\overline{\lambda}}}{2}/\pmb{\lambda}\in \mathcal{S}^{+}\}=\{[r,\cos(\theta)]/0\neq r\in \mathbb{R}^{+}, \theta\in ]0,\pi[\}$, where $[r,\cos(\theta)]=(r^{k}\cos(k\theta))_{k}$
\item $\mathcal{S}^{+}_{2}=\{\frac{\pmb{\lambda}-\pmb{\overline{\lambda}}}{2\mathrm{i}}/\pmb{\lambda}\in \mathcal{S}^{+}\}=\{[r,\sin(\theta)]/0\neq r\in \mathbb{R}^{+}, \theta\in ]0,\pi[\}$, where $[r,\sin(\theta)]=(r^{k}\sin(k\theta))_{k}$
    \end{itemize}
Clearly $\widetilde{S}, \mathcal{S}^{+}, \mathcal{S}^{-}$ constitute a partition of $\mathcal{S}^{\ast}$, and then the set $\mathcal{S}^{\circ}\cup \mathbb{S}$, $\mathbb{S}=\widetilde{S}\bigcup \mathcal{S}^{+}_{1}\bigcup \mathcal{S}^{+}_{2}$, is $\mathbb{C}$-linearly independent and thus is $\mathbb{R}$-linearly independent.
\\From the the general theory of linear recurrence sequences we have the following:
\begin{itemize}
\item $\mathbb{R}_{\mathcal{S}^{\circ}}\bigoplus \mathbb{R}_{\mathbb{S}}$ is the $\mathbb{R}$-vector space of real linear recurrence sequences whose characteristic polynomials are of the form $X^{m}P(X)$ where the polynomials $P(X)\in \mathbb{R}[X]$ are square-free with nonzero constant terms.
\item $\pmb{\Gamma}$ is transcendental over $\mathbb{C}_{\mathcal{S}^{\ast}}$ and it is so over $\mathbb{R}_{\mathbb{S}}$
\item $\mathbb{R}_{\mathcal{S}^{\circ}}\bigoplus \mathbb{R}_{\mathbb{S}}[\Gamma]$ is the $\mathbb{R}$-vector space of real linear recurrence sequences whose characteristic polynomials $P(X)\in \mathbb{R}[X]$.
\item If $\pmb{u}\in \mathbb{R}_{\mathcal{S}^{\circ}}\bigoplus \mathbb{R}_{\mathbb{S}}[\Gamma]$ then there exist $\pmb{\rho}_{1},\ldots,\pmb{\rho}_{l}\in \widetilde{S}$, $\pmb{\lambda}_{1}=(r_{1}^{k}e^{\mathrm{i}k\theta_{1}})_{k},\ldots,\pmb{\lambda}_{m}=(r_{m}^{k}e^{\mathrm{i}k\theta_{m}})_{k}\in \mathcal{S}^{+}$, $\theta_{1},\ldots,\theta_{m}\in ]0,\pi[$, $\mu_{1},\ldots,\mu_{q}\in \mathbb{R}$, $P_{1},\ldots, P_{l}\in \mathbb{R}[X]$ and $Q_{1},\ldots, Q_{m}\in \mathbb{C}[X]$ such that, for all $k\geq 0$,
   \begin{eqnarray*}
    u(k)&=&\sum_{i=1}^{q}\mu_{i}\pmb{0}_{i}(k)+\sum_{i=0}^{l}P_{i}(k)\rho_{i}^{k}+
    \sum_{j=0}^{m}(Q_{j}(k)\lambda_{j}^{k}+\overline{Q}_{j}(k)\overline{\lambda}_{j}^{k})\\
    &=&\sum_{i=1}^{q}\mu_{i}\pmb{0}_{i}(k)+\sum_{i=0}^{l}P_{i}(k)\rho_{i}^{k}+
    \sum_{i=0}^{m}(Q_{i}(k)\lambda_{i}^{k}+\overline{Q_{i}(k)\lambda_{i}^{k}})\\
    &=&\sum_{i=1}^{q}\mu_{i}\pmb{0}_{i}(k)+\sum_{i=0}^{l}P_{i}(k)\rho_{i}^{k}+
    \end{eqnarray*}
    \begin{eqnarray*}
    &&\sum_{i=0}^{m}2Re(Q_{i}(k))[r_{i},\cos(\theta_{i})](k)-\sum_{i=0}^{m}2Im(Q_{i}(k))[r_{i},\sin(\theta_{i})](k).
    \end{eqnarray*}
This means that $\mathcal{S}^{\circ}\bigcup \{\pmb{\lambda}\pmb{\Gamma}^{i}/0\neq \lambda\in \mathbb{R}, i\in \mathbb{N}\}\bigcup \{[r,\cos(\theta)]\pmb{\Gamma}^{i}, [r,\sin(\theta)]\pmb{\Gamma}^{i}/ i\in \mathbb{N},\,\,0\neq r\in \mathbb{R}^{+},\,\, \theta\in ]0,\pi[\}$ spans the $\mathbb{R}$-vector space $\mathbb{R}_{\mathcal{S}^{\circ}}\bigoplus \mathbb{R}_{\mathbb{S}}[\pmb{\Gamma}]$, and hence is a basis for it.
\end{itemize}
With the above results and notations we can formulate the following.
\begin{thm}\label{Thmmm 1002}  Let $A\in M_{q}(\mathbb{R})$ be a real matrix and let
$\mathbb{S}=\widetilde{S}\bigcup \mathcal{S}^{+}_{1}\bigcup \mathcal{S}^{+}_{2}$.
Then the sequence of matrices obtained by plugging $-k$ for $k$ in the geometric part of $A$ relative to $(\mathbb{S}, \mathcal{H})$ is $\pmb{A}_{d}-\pi_{0}\pmb{0}_{0}$.
\end{thm}
\proof The proof follows simply from the fact that the equalities $$\pmb{\lambda}+\pmb{\overline{\lambda}}=2[r,\cos(\theta)]$$ and $$\pmb{\lambda}-\pmb{\overline{\lambda}}=2\mathrm{i}[r,\sin(\theta)]$$ still valid when we plug in $-k$ for $k$.
\endproof
From Theorem~\ref{Thmmm 1002}, we obtain the following theorem.
\begin{prop}  Let $A\in M_{q}(\mathbb{R})$ be a real matrix and let
$\mathbb{S}=\widetilde{S}\bigcup \mathcal{S}^{+}_{1}\bigcup \mathcal{S}^{+}_{2}$.
Then the sequence of matrices obtained by plugging $-k$ for $k$ in the geometric part of $A$ relative to $(\mathbb{S}, \mathcal{T})$ is $\pmb{A}_{d}-\pi_{0}\pmb{0}_{0}$.
\end{prop}
\proof The proof follows immediately from the fact that $\mathcal{T}$ is contained in the $\mathbb{R}$-vector space spanned by $\mathcal{H}=(\pmb{\Gamma}^{i})_{i\geq 0}$, since $(\pmb{\Lambda}_{i})_{i\geq 0}=D(\pmb{\Gamma}^{i})_{i\geq 0}$ where
$$D=(\frac{s(n,k)}{n!})_{n,k\geqslant 0}$$ is the infinite lower triangular matrices with coefficients in $\mathbb{R}$ and $s(n,k)$ are the Stirling numbers of the first kind (see, e.g. Quaintance and Gould~\cite{Qua}).
\endproof
\begin{rmk} In the case of a field of characteristic $0$, the greatest integer such that $\pmb{\lambda}_{j}\pmb{\Lambda}_{t_{j}-1}$ appear in the $\mathcal{P}$-cf of $A$ relative to $(\mathcal{S}^{\ast}, \mathcal{T})$ and the greatest integer such that $\pmb{\lambda}_{j}\pmb{\Gamma}^{t_{j}-1}$ appear in the $\mathcal{P}$-cf of $A$ relative to $(\mathcal{S}^{\ast}, \mathcal{H})$ are the same. Therefore in Corollary~\ref{Cor 22}, we can replace $\mathcal{T}$ with $\mathcal{H}$ and $\pmb{\lambda}_{j}\pmb{\Lambda}_{t_{j}-1}$ with $\pmb{\lambda}_{j}\pmb{\Gamma}^{t_{j}-1}$.
\end{rmk}
\begin{rmk}\label{remark13} It is obvious that  the $0$th coordinate $\pmb{\mathcal{A}}_{0}$ of $A$ with respect to $(\mathcal{S}^{\ast}, \mathcal{T})$ and the $0$th coordinate $\pmb{\mathcal{A'}}_{0}$ of $A$ with respect to $(\mathcal{S}^{\ast}, \mathcal{H})$ are the same, then $A_{s}=\pmb{\mathcal{A'}}_{0}(1)$ in virtue of Corollary~\ref{Cor 11}.
\\In general we can easily establish the following formulas for change of coordinates of a matrix $A$ from $(\mathcal{S}^{\ast}, \mathcal{T})$ to $(\mathcal{S}^{\ast}, \mathcal{H})$, and vice versa:
$$\pmb{\mathcal{A}}_{i}=\sum_{j=0}^{m}j!S(j,i)\pmb{\mathcal{A'}}_{j}$$
and
$$\pmb{\mathcal{A'}}_{i}=\sum_{j=0}^{m}\frac{s(j,i)}{j!}\pmb{\mathcal{A}}_{j}$$
where $S(n,j)$ and $s(n,j)$ are the Stirling numbers of the second kind and the first kind respectively, and $m$ is the integer defined in
Proposition~\ref{Prop 23}.
\end{rmk}
For the purpose of illustration, let us consider the following examples.
\begin{ex} Let
$$B=\begin{pmatrix} 1&1&0&0\\
-2&0&1&0\\2&0&0&1\\-2&-1&-1&-1
\end{pmatrix}.$$
Then one can verify that for all $k\geq 1$ we have $B^{k}=B(k)$, where
\begin{equation*}
B(k)=\left(\begin{array}{ccc}
\cos(\dfrac{k\pi}{2})+\sin(\dfrac{k\pi}{2})&\dfrac{2\sin(\dfrac{k\pi}{2})-k\cos(\dfrac{k\pi}{2})}{2}\\
-2\sin(\dfrac{k\pi}{2})&\dfrac{(k+2)\cos(\dfrac{k\pi}{2})+(k-1)\sin(\dfrac{k\pi}{2})}{2}\\
2\sin(\dfrac{k\pi}{2})&\dfrac{-k\cos(\dfrac{k\pi}{2})+(1-k)\sin(\dfrac{k\pi}{2})}{2}\\
-2\sin(\dfrac{k\pi}{2})&
\dfrac{k\cos(\dfrac{k\pi}{2})+(k-3)\sin(\dfrac{k\pi}{2})}{2}
\end{array}\right.
\end{equation*}
\begin{equation*}
\hspace*{2.5cm}\left.\begin{array}{ccc}
\dfrac{(1-k)\sin(\dfrac{k\pi}{2})-k\cos(\dfrac{k\pi}{2})}{2}&\dfrac{(1-k)\sin(\dfrac{k\pi}{2})}{2}\\
k\sin(\dfrac{k\pi}{2})&
\dfrac{-k\cos(\dfrac{k\pi}{2})+(k-1)\sin(\dfrac{k\pi}{2})}{2}\\
(1-k)\sin(\dfrac{k\pi}{2})+\cos(\dfrac{k\pi}{2})&
\dfrac{k\cos(\dfrac{k\pi}{2})+(3-k)\sin(\dfrac{k\pi}{2})}{2}\\
(k-2)\sin(\dfrac{k\pi}{2})&
\dfrac{(2-k)\cos(\dfrac{k\pi}{2})+(k-3)\sin(\dfrac{k\pi}{2})}{2}
\end{array}\right)
\end{equation*}
\\From this result, we may conclude the following:
\begin{itemize}
\item Since $B(0)=I_{4}$, $B$ is nonsingular.
\item
$$B^{-1}=B(-1)=\begin{pmatrix} -1&-1&-1&-1\\2&1&1&1\\-2&-1&-2&-2\\2&2&3&2\end{pmatrix}.$$
\item Also we have $B^{-k}=B(-k)$.
\item Using Corollary~\ref{Cor 22}, we get that the minimal polynomial of $B$ is
$$(X-\e^{\frac{\pi\mathrm{i}}{2}})^{2}(X-\e^{\frac{-\pi\mathrm{i}}{2}})^{2}=(X^{2}+1)^{2}$$
\item By virtue of Corollary~\ref{Cor 11} and Remark~\ref{remark13}, we have
$$B_{s}=\pmb{\mathcal{B}}_{0}(1) =\begin{pmatrix} 1&1&\frac{1}{2}&\frac{1}{2}\\
-2&\frac{-1}{2}&0&\frac{-1}{2}\\2&\frac{1}{2}&1&\frac{3}{2}\\-2&\frac{-3}{2}&-2&\frac{-3}{2}
\end{pmatrix}$$
\end{itemize}
\end{ex}

\begin{ex} Let us take example~\ref{example 1}. We have
$$A_{d}=A(-1)=\begin{pmatrix}
\frac{1}{4}&\frac{1}{4}&\frac{3}{16}&\frac{-1}{16}\vspace*{0.1pc}\\
\frac{1}{4}&\frac{1}{4}&\frac{5}{16}&\frac{-3}{16}\vspace*{0.1pc}\\
0&0&\frac{-1}{4}&\frac{1}{4}\vspace*{0.1pc}\\
0&0&\frac{1}{4}&\frac{-1}{4}
\end{pmatrix}.$$
For all $k\geq 1$,
$$A_{d}^{k}=A(-k)=\begin{pmatrix}
2^{-1-k}&2^{-1-k}&\frac{1}{16}2^{-k}((-1)^{1+k}+5)&\frac{1}{16}2^{-k}((-1)^{k}-1)\vspace*{0.1pc}\\
2^{-1-k}&2^{-1-k}&\frac{5}{16}2^{-k}((-1)^{1+k}+1)&\frac{1}{16}2^{-k}(5(-1)^{k}-1)\vspace*{0.1pc}\\
0&0&(-1)^{k}2^{-1-k}&(-1)^{1+k}2^{-1-k}\vspace*{0.1pc}\\
0&0&(-1)^{1+k}2^{-1-k}&(-1)^{k}2^{-1-k}
\end{pmatrix}.$$
\end{ex}
\begin{ex} Let $x\in \mathbb{C}$ and let
$$E=\begin{pmatrix}
2\sqrt{3}-x-10&2\sqrt{3}-2x-23&\sqrt{3}-x-5\\4&\sqrt{3}+9&2\\
-2\sqrt{3}+2x+2&-4\sqrt{3}+4x+5&-\sqrt{3}+2x+1
\end{pmatrix}$$
Let
\begin{eqnarray*}
\delta_{s}(x)=\begin{cases}
0 &\text{if}\quad x=0\\x^{s} &\text{if}\quad x\neq 0
\end{cases}
\end{eqnarray*}
and put
$$E(k)=
\begin{pmatrix}
e_{11}(k)&e_{12}(k)&e_{13}(k)\\e_{21}(k)&e_{22}(k)&e_{23}(k)\\e_{31}(k)&e_{32}(k)&e_{33}(k)
\end{pmatrix}$$
where
\begin{eqnarray*}
e_{11}(k)&=&2^{k+1}(\cos(\frac{k\pi}{6})-5\sin(\frac{k\pi}{6}))-\delta_{k}(x)\\
e_{12}(k)&=&2^{k+1}(\cos(\frac{k\pi}{6})-\frac{23}{2}\sin(\frac{k\pi}{6}))-2\delta_{k}(x)\\
e_{13}(k)&=&2^{k}(\cos(\frac{k\pi}{6})-5\sin(\frac{k\pi}{6}))-\delta_{k}(x)\\
e_{21}(k)&=&2^{k+2}\sin(\frac{k\pi}{6})\\
e_{22}(k)&=&2^{k}(\cos(\frac{k\pi}{6})+9\sin(\frac{k\pi}{6}))\\
e_{23}(k)&=&2^{k+1}\sin(\frac{k\pi}{6})\\
e_{31}(k)&=&-2^{k+1}(\cos(\frac{k\pi}{6})-\sin(\frac{k\pi}{6}))+2\delta_{k}(x)\\
e_{32}(k)&=&-2^{k+2}(\cos(\frac{k\pi}{6})-\frac{5}{4}\sin(\frac{k\pi}{6}))+4\delta_{k}(x)\\
e_{33}(k)&=&-2^{k}(\cos(\frac{k\pi}{6})-\sin(\frac{k\pi}{6}))+2\delta_{k}(x)
\end{eqnarray*}
Then one can verify that for all $k\geq1$, we have
$E^{k}=E(k)$.
\\From this result one can conclude easily the following:
\begin{itemize}
\item $$E(0)=\begin{pmatrix} 2-\delta_{0}(x)&2-2\delta_{0}(x)&1-\delta_{0}(x)\\0&1&0\\-2+2\delta_{0}(x)&-4+4\delta_{0}(x)&-1+2\delta_{0}(x)\end{pmatrix}$$
Thus $E$ is nonsingular if and only if $x\neq 0$.
\item $$E_{d}=\begin{pmatrix} \frac{\sqrt{3}+5}{2}-\delta_{-1}(x)&\frac{2\sqrt{3}+23}{4}-2\delta_{-1}(x)&\frac{\sqrt{3}+5}{4}-\delta_{-1}(x)\vspace*{0.1pc}\\
-1&\frac{\sqrt{3}-9}{4}&\frac{-1}{2}\vspace*{0.1pc}\\\frac{-\sqrt{3}-1}{2}+2\delta_{-1}(x)&\frac{-4\sqrt{3}-5}{4}+4\delta_{-1}(x)&
\frac{-\sqrt{3}-1}{4}+2\delta_{-1}(x)\end{pmatrix}$$
\item  For all $k\geq1$
$$E_{d}^{k}=\begin{pmatrix}
a_{11}(k)&a_{12}(k)&a_{13}(k)\\a_{21}(k)&a_{22}(k)&a_{23}(k)\\a_{31}(k)&a_{32}(k)&a_{33}(k)
\end{pmatrix}$$
where
\begin{eqnarray*}
a_{11}(k)&=&2^{-k+1}(\cos(\frac{k\pi}{6})+5\sin(\frac{k\pi}{6}))-\delta_{-k}(x)\\
a_{12}(k)&=&2^{-k+1}(\cos(\frac{k\pi}{6})+\frac{23}{2}\sin(\frac{k\pi}{6}))-2\delta_{-k}(x)\\
a_{13}(k)&=&2^{-k}(\cos(\frac{k\pi}{6})+5\sin(\frac{k\pi}{6}))-\delta_{-k}(x)\\
a_{21}(k)&=&-2^{-k+2}\sin(\frac{k\pi}{6})\\
a_{22}(k)&=&2^{-k}\cos(\frac{k\pi}{6})-9\times2^{-k}\sin(\frac{k\pi}{6})\\
a_{23}(k)&=&-2^{-k+1}\sin(\frac{k\pi}{6})\\
a_{31}(k)&=&-2^{-k+1}(\cos(\frac{k\pi}{6})+\sin(\frac{k\pi}{6}))+2\delta_{-k}(x)\\
a_{32}(k)&=&-2^{-k+2}(\cos(\frac{k\pi}{6})-\frac{5}{4}\sin(\frac{k\pi}{6}))+4\delta_{-k}(x)\\
a_{33}(k)&=&-2^{-k}(\cos(\frac{k\pi}{6})+\sin(\frac{k\pi}{6}))+2\delta_{-k}(x)
\end{eqnarray*}
\item Using Corollary~\ref{Cor 22}, we get that the minimal polynomial of $E$ is
$$
(X-x)(X-2\e^{\frac{\pi\mathrm{i}}{6}})^{2}(X-\e^{\frac{-\pi\mathrm{i}}{6}})^{2}=(X-x)(X^{2}-2\sqrt{3}X+4)
$$
\end{itemize}
\end{ex}


\begin{thebibliography}{99}

\bibitem{Ben} A. Ben-Israel and T.N.E. Greville. \emph{Generalized inverse Theory and Applications.} 2nd Edition, New York, Springer Verlag, (2003).

\bibitem{Camp3} S.L. Campbell et al. \emph{Applications of the Drazin Inverse to Linear Systems of Differential Equations with Singular Constant Coefficients.} SIAM J. Appl. Math. Vol., {\bf 31(3)}, 411-425, 1976.

\bibitem{Camp2} S.L. Campbell. \emph{Continuity of The Drazin inverse.} Linear and Multilinear Algebra, {\bf 8}, 265-268, 1980.

\bibitem{Camp1} S.L. Campbell. \emph{The Drazin inverse and systems of second order linear differential equations.} Linear and Multilinear Algebra, {\bf 14(2)}, 195-198, 1983.

\bibitem{Camp} S.L. Campbell and C. D. Meyer, Jr. \emph{Generalized inverses of linear transformations.} Dover, New York, 1991.

\bibitem{Cohn} P.M. Cohn. R. \emph{Basic Algebra.} Springer, London, (2003).

\bibitem{Cvet} D.S. Cvetkovi\'{c}-Ili\'{c}, J. Chen, and Z. Xu. \emph{Explicit representations of the Drazin inverse of block matrix and modified matrix.} Linear and Multilinear Algebra, {\bf 57(4)}, 355-364, 2009.

\bibitem{Draz}  M.R. Drazin. \emph{Pseudo-inverses in associative rings and semigroups.}
Amer. Math. Monthly, {\bf 65}, 506-514, 1958.

\bibitem{Fill} J.P. Fillmore and M.L. Marx. \emph{Linear recursive sequences.} SIAM Rev., {\bf 10}, 342-353, 1968.

\bibitem{Grev} T.N.E. Greville. \emph{The Souriau-Frame Algorithm and the Drazin Pseudoinverse.} Linear Algebra Appl., {\bf 6}, 205-208, 1973.

\bibitem{Kwap} M. Kwapisz, \emph{Remarks on the calculation of the power of a matrix.} J. Difference Equ. Appl., {\bf 10(2)}, 139-149, 2004.

\bibitem{Kyrc} I.I. Kyrche. \emph{Analogs of the adjoint matrix for generalized inverse and corresponding Cramer rules.} Linear and Multilinear Algebra, {\bf 56(4)}, 453-469, 2008.

\bibitem{Lev} J. Levine and R.E. Hartwig. \emph{Applications of Drazin inverse to the Hill cryptographic systems.} Cryptologia, {\bf 4}, 71-85, 1980.

\bibitem{Liu} X. Liu, L. Wu, and J. Benitez. \emph{On linear combinations of generalized involutive matrices.} Linear and Multilinear Algebra, {\bf 59(11)}, 1221-1236, 2011.

\bibitem{Male} B. Malesevic and I. Jovovic. \emph{A procedure for finding the k-th power of a matrix.} Maplesoft, (accessed July 05, 2020), 2007.

\bibitem{Mou1} M. Mou\c{c}ouf. \emph{Arbitrary positive powers of semicirculant and $r$-circulant matrices.} Linear and
Multilinear Algebra. 2021. DOI:10.1080/03081087.2021.1968329.

\bibitem{Mou2} M. Mou\c{c}ouf. \emph{A closed-form expression for the $k$th power of semicirculant and $r$-circulant matrices.} arXiv:2006.16198v1 [math.RA].

\bibitem{Qua} J. Quaintance and H.W. Gould. \emph{Combinatorial identities for Stirling numbers: the unpublished notes of HW Gould.} World Scientific, (2015).

\bibitem{Riord} J. Riordan. \emph{Combinatorial Identities.} John Wiley $\&$ Sons, Inc, New York, (1968).

\bibitem{Ros} N. Rose. \emph{A Note on Computing the Drazin Inverse.} Linear Algebra
Appl., {\bf 15}, 95-98, 1976.

\bibitem{Sol} F. Soleymani and P.S Stanimirovi\'{c}. \emph{A higher order iterative method for computing the Drazin inverse.} Sci. World J., 115-124, 2013.

\bibitem{Wei} Y. Wei. \emph{Expression for the Drazin inverse of a $2 \times 2$ block matrix.} Linear and Multilinear Algebra, {\bf 45(2-3)}, 131-146, 1998.

\bibitem{Wilk} J. Wilkinson. \emph{Note on the Practical Significance of the Drazin Inverse.} Campbell, S.L., Ed.; Recent Applications of Generalized Inverses, Pitman Advanced Publishing Program, Research Notes in Mathematics, No. 66,Boston; NASA: Washington, DC, USA, 82-99, 1982.

\bibitem{Yu} Y. Yu and Y. Wei. \emph{Determinantal representation of the generalized inverse $A^{(2)}_{T,S}$ over integral domains and its applications.} Linear and Multilinear Algebra {\bf 57}, 547-559, 2009.

\bibitem{Zhang} X. Zhang and G. Chen. \emph{The computation of Drazin inverse and its application in Markov chains.} Appl. Math. Comput., {\bf183}, 292-300, 2006.

\end{thebibliography}
\end{document}